\documentclass[12pt, reqno]{amsart}
\usepackage[top=3cm, bottom=3.5cm, left=3cm, right=3cm]{geometry}
\usepackage{amsmath,amsfonts,amssymb,amsthm}
\usepackage{mathtools}
\usepackage[utf8]{inputenc}
\usepackage{hyperref}
\usepackage{mathrsfs}
\usepackage{tikz-cd}
\usepackage{lipsum}
\usepackage{csquotes}
\usepackage{hyperref}

\def\R{\operatorname{\mathbb{R}}}

\usepackage{geometry}
\usepackage{booktabs}
\usepackage{array}
\usepackage{paralist}
\usepackage{verbatim}
\usepackage{subfig}
\usepackage{fancyhdr}
\author{Skilyn Leon}
\address{Department of Mathematics\\
  CUNY City College\\
  New York, NY 10031, USA}
\email[S. Leon]{skilyn.leon@mail.citytech.cuny.edu}
\author{Rosa Pavlak}
\address{Department of Mathematics\\
  CUNY City Tech\\
  New York, NY 11201, USA}
\email[R. Pavlak.]{rosa.pavlak@mail.citytech.cuny.edu}
\author{Evelyn Pulla}
\address{Department of Mathematics\\
  CUNY City Tech\\
  New York, NY 11201, USA}
\email[E. Pulla]{evelyn.pulla28@citytech.cuny.edu}
\author{Xi Sisi Shen*}
\thanks{*Supported in part by NSF grant DMS-2601275.}
\address{Department of Mathematics\\
  CUNY City Tech\\
  New York, NY 11201, USA}
\email[X. S. Shen]{xi.shen15@citytech.cuny.edu}

\newtheorem{proposition}{Proposition}
\newtheorem{thm}{Theorem}

\newtheorem{lemma}{Lemma}

\numberwithin{equation}{section}

\title[Four-arc regions for a strip density]{A comparison theorem for the planar strip isoperimetric problem}
\date{}
\begin{document}
\begin{abstract}
We consider the isoperimetric problem with planar density equal to 1 on a horizontal strip and to a constant $\lambda>1$ outside the strip. We resolve a conjecture made by Ca\~nete-Miranda Jr-Vittone by showing that every four-arc region admits a three-arc competitor with the same weighted area and strictly smaller weighted perimeter. The argument applies to every $\lambda>1$ and gives an explicit positive lower bound for the perimeter difference. 
\end{abstract}
\maketitle

\section{Introduction}
The isoperimetric problem of enclosing a given amount of area with the shortest boundary is one of the oldest questions in geometry, but the tools to prove it rigorously are rather recent. De Giorgi's theory of sets of finite perimeter, developed in the 1950s gave the first proof that the round ball minimizes perimeter for its volume in $\mathbb{R}^n$ \cite{DG1,DG2}. That same machinery turned out to generalize far beyond Euclidean space, by weighting the usual volume and surface area by integrating a density function led to the theory of manifolds with density. The best-known example is Gauss space, where the isoperimetric sets are half-spaces \cite{Borell}. A systematic isoperimetric theory for general densities on $\mathbb{R}^n$ led to the Log-Convex Density Conjecture of Rosales, Cañete, Bayle and Morgan \cite{RCBM} which states that for a smooth, radial, log-convex density, balls about the origin are always isoperimetric. This conjecture was fully resolved only in 2015, by Chambers \cite{Chambers}. Piecewise constant densities are less well understood. Cañete, Miranda Jr. and Vittone \cite{CMV} gave the first systematic study of such densities, including the strip density
$$f_\lambda(x,y) = \begin{cases} 1, & |y|\le 1,\\ \lambda, & |y|>1,\end{cases} \qquad \lambda > 1.$$
This is the simplest density with a genuine jump discontinuity: a minimizer must trade curvature against a change in the cost of boundary length as it crosses a fixed line and this gives rise to a Snell refraction law for the crossing angles as was shown in \cite{CMV}.

For a bounded region $E\subset\R^2$ with piecewise smooth boundary, write
\[
 \mathcal{A}_\lambda(E)=\int_E f_\lambda\,dx\,dy,
 \qquad \mathcal{P}_\lambda(E)=\int_{\partial E}f_\lambda\,d\mathcal{H}^1.
\]
The values of the density on the two interfaces are therefore equal to one. All regions used below have piecewise smooth Jordan boundaries.

Ca\~nete, Miranda Jr.\ and Vittone~\cite{CMV} describe four geometric families arising in the strip-density problem and establish that these are the only possible isoperimetric candidates. They prove in Proposition 3.9 of \cite{CMV} that these candidates must have vertical reflective symmetry. The four candidates are:
\begin{enumerate}[(i)]
    \item balls fully contained in the strip
    \item pieces of the strip bounded on the left and right sides by semicircles and bounded from the top and bottom by $\{|y|=1\}$
    \item sets bounded by three circular arcs with the same radius and a straight line segment (possibly a single point) contained in $\{y=-1\}$ on the bottom. The top arc is contained in $\{y>1\}$ and the other two are contained in the strip meeting the top arc at an angle determined by Snell's refraction law and meeting the bottom straight line segment tangentially.
    \item sets bounded by four circular arcs with the same radius with two of them contained within the strip, a top arc in $\{y>1\}$ and a bottom arc in $\{y<1\},$ meeting at angles determined by Snell's law.
\end{enumerate}

\begin{figure}[htbp]
    \centering 
    \includegraphics[width=0.8\textwidth]{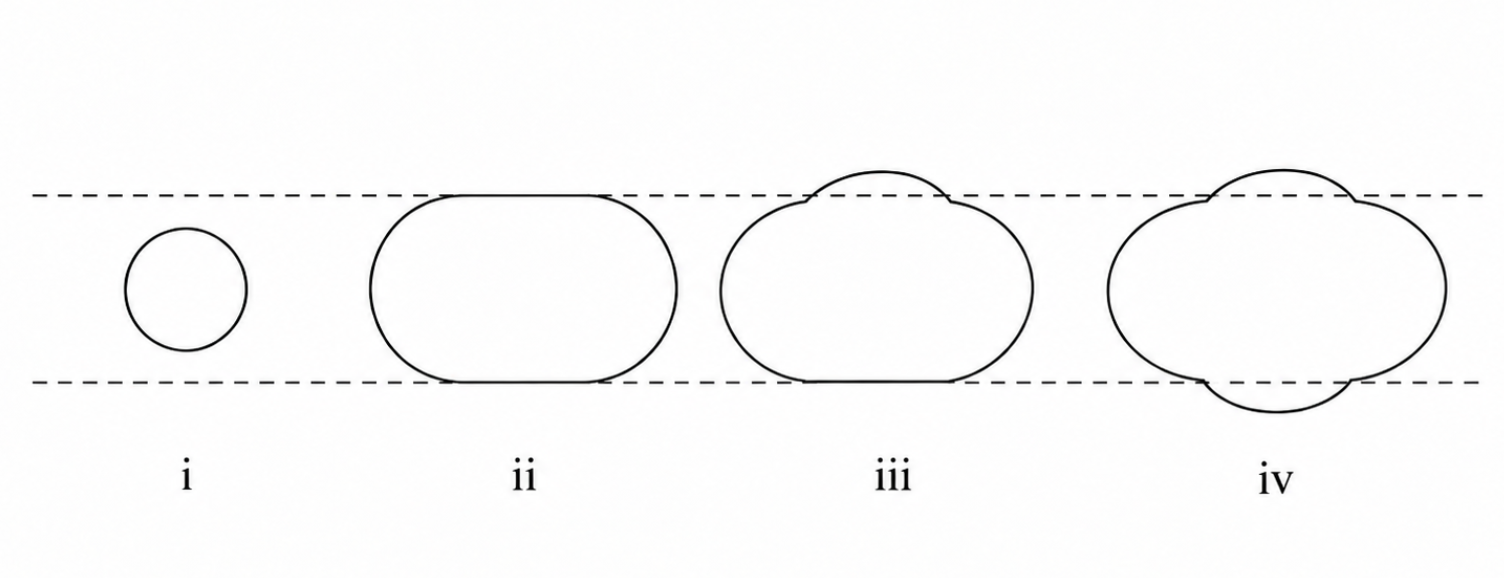} 
    \caption{The four isoperimetric candidates.}
    \label{fig:my_image} 
\end{figure}

In Conjecture 3.12 of \cite{CMV}, it was conjectured whether the four-arc candidate could ever be isoperimetric, and was shown for values of $\lambda>\pi/4$, that it cannot be. In this paper, we extend this result to full generality for all $\lambda>1$. We consider the symmetric four-arc family, consisting of two circular arcs in the strip and two congruent exterior caps. We denote its the four-arc region with curvature $h$ by $C_h$. Our main theorem is as follows:

\begin{thm}\label{thm:main}
For every $\lambda>1$ and $0<h\leq1$, there is a three-arc region $E_s$, with $0<s<h$, such that
\[
 \mathcal{A}_\lambda(E_s)=\mathcal{A}_\lambda(C_h)
\]
and
\begin{equation}\label{eq:quantitative}
 \mathcal{P}_\lambda(C_h)-\mathcal{P}_\lambda(E_s)>
 \frac{\gamma_\lambda}{h}>0,
 \qquad
 \gamma_\lambda:=\lambda\arccos(1/\lambda)-\sqrt{1-\lambda^{-2}}.
\end{equation}
Consequently, $C_h$ cannot minimize weighted perimeter at fixed weighted area.
\end{thm}

The theorem concerns the explicit symmetric family above. Its proof does not use a classification theorem for arbitrary minimizers. The accompanying source repository~\cite{Repository} also contains an all-density exclusion for this model. Our purpose here is to give a direct proof of the comparison, with the geometric realization and endpoint arguments included.

The paper is organized as follows: in Section \ref{sec:geometry} we give explicit expressions for the functions that govern the arcs of the three- and four-arc families and show the area and perimeter functions in terms of curvature, in Section \ref{sec:comparison}, we prove that for every 4-arc region, there exists a 3-arc region with equal area and strictly smaller perimeter and establish the explicit perimeter reduction achieved.

\section*{Acknowledgments}
We would like to thank Frank Morgan and David Thompson for sharing with us their detailed report on ``Isometric Problems in Sectors with Density" which provided us with useful insight into proof techniques for related isoperimetric problems. This project was completed as part of an undergraduate research program at CUNY City Tech during the summer of 2026, and the first three authors were given summer support from NSF grant DMS-2601275.

\section*{Declaration of AI Use}
A bespoke agentic orchestration program, \href{https://github.com/RosaRojacr/agentic-lean-math-assistant}{Agentic Lean Math Assistant}, was developed to investigate this and related problems, with GPT-6-Astra as the coordinating model and GPT-5.6-Sol, GPT-5.6-Terra, and GPT-5.6-Luna as subagents. The program assisted in discovering intermediate results, developing proof strategies, and formalizing the resulting arguments in Lean. The formalization passed Lean’s kernel checks, and a separate AI-assisted semantic audit found the formal statements consistent with their mathematical exposition. The authors independently checked the mathematical arguments presented here and wrote the final manuscript in their own words. The accompanying source repository \cite{Repository} contains an all-density exclusion for this model. 

\section{Geometric formulas}\label{sec:geometry}
In this section, we will introduce the notation and formulas we will use for our comparison theorem of the three- and four-arc regions needed for the proof in our main theorem.

Fix $\lambda>1$. For $-1\leq z\leq1$, define
\begin{equation}\label{eq:BD}
 B(z)=\lambda\arccos(z/\lambda)+\arcsin z,
 \qquad
 D(z)=\sqrt{1-z^2/\lambda^2}-\sqrt{1-z^2}.
\end{equation}
For $0<s\leq1$, let $q=2s-1$ and set
\begin{align}
 A_3(s)&=\frac{B(q)+\pi/2+(2s+1)D(q)}{s^2},
 &P_3(s)&=\frac{2\bigl(B(q)+\pi/2+D(q)\bigr)}{s},\label{eq:three}\\
 A_4(s)&=\frac{2\bigl(B(s)+sD(s)\bigr)}{s^2},
 &P_4(s)&=\frac{4B(s)}{s}.\label{eq:four}
\end{align}
These quantities are equivalent to those in \cite{CMV} through the use of the trigonometric identity $\arcsin(2h-1)-\frac{\pi}{2}=2\arcsin(\sqrt{h})$. The values $A_3(1)$ and $P_3(1)$ will be used only as continuous endpoint values of the formulas. The competitors $E_s$ constructed in the proof always have $s<1$.

We first verify that these are the weighted areas and perimeters of the relevant regions. For $0<h\leq1$, let
\[
 R=h^{-1},\qquad \alpha=\arccos(h/\lambda),\qquad d=D(h)/h.
\]
Define the central region and upper cap by
\begin{align*}
 K_h&=\{(x,y): -1\leq y\leq1,\ |x|\leq d+\sqrt{R^2-y^2}\},\\
 H_h&=\{(x,y):y\geq1,\ x^2+(y-(1-R\cos\alpha))^2\leq R^2\}.
\end{align*}
If $\sigma(x,y)=(x,-y)$, then
\[
 C_h=K_h\cup H_h\cup\sigma(H_h).
\]
At $y=\pm1$, the half-width of $K_h$ is $\sin\alpha/h$, so the caps and central region meet along their full chords. Interior points of these chords are interior points of the union and contribute no boundary length. The exposed boundary consists of the two strip arcs and two exterior arcs. Hence
\[
 \mathcal{P}(C_h)=\frac{4\arcsin h}{h}+\frac{4\lambda\alpha}{h}=P_4(h).
\]
The central area and the area of one exterior cap are
\[
 |K_h|=\frac{4D(h)}h+\frac{2\sqrt{1-h^2}}h+
 \frac{2\arcsin h}{h^2},
 \qquad
 |H_h|=\frac{\alpha-\sin\alpha\cos\alpha}{h^2}.
\]
It follows that
\begin{equation}\label{eq:four-realized}
 \mathcal{A}(C_h)=A_4(h).
\end{equation}
These computations remain valid when $h=1$.

For the three-arc family, let $0<s<1$ and set
\[
 q=2s-1,\qquad R=s^{-1},\qquad
 \alpha=\arccos(q/\lambda),\qquad b=D(q)/s.
\]
Since $D(q)\geq0$, we have $b\geq0$. Define
\begin{align*}
 K_s^{(3)}&=\{(x,y):-1\leq y\leq1,\
 |x|\leq b+\sqrt{R^2-(y-(R-1))^2}\},\\
 H_s^{(3)}&=\{(x,y):y\geq1,\
 x^2+(y-(1-R\cos\alpha))^2\leq R^2\},
\end{align*}
and put $E_s=K_s^{(3)}\cup H_s^{(3)}$. The radicand in $K_s^{(3)}$ is nonnegative throughout the strip. Its upper half-width is
\[
 b+\frac{\sqrt{1-q^2}}s=\frac{\sin\alpha}s,
\]
which agrees with the cap half-chord. Its lower boundary is the exposed segment of length $2b$ on $y=-1$.

The two strip arcs have total length $2(\arcsin q+\pi/2)/s$, and the exterior arc has length $2\alpha/s$. Therefore
\[
 \mathcal{P}(E_s)=\frac{2\lambda\alpha}{s}
  +\frac{2(\arcsin q+\pi/2)}s+2b=P_3(s).
\]
Direct integration gives
\[
 |K_s^{(3)}|=\frac{4D(q)}s+
 \frac{\arcsin q+\pi/2+q\sqrt{1-q^2}}{s^2},
 \qquad
 \lambda|H_s^{(3)}|=\frac{\lambda\alpha-q\sin\alpha}{s^2}.
\]
Adding these terms yields
\begin{equation}\label{eq:three-realized}
 \mathcal{A}(E_s)=A_3(s),\qquad \mathcal{P}(E_s)=P_3(s).
\end{equation}

When $s<1/2$, the upper cap has $\alpha>\pi/2$ and is a major circular segment. It is still given by the intersection of the disk with $y\geq1$, so the construction remains valid. When $s=1/2$, the lower segment reduces to one point and $E_s$ is the disk of radius two centered at $(0,1)$. Thus every parameter $0<s<1$ gives a valid competitor. All the regions above are bounded, and the finitely many arc junctions have zero $\mathcal{H}^1$-measure.

\section{Two comparisons at the same curvature}\label{sec:comparison}
The proof will use two inequalities: $A_3(h)<A_4(h)$ and a strict comparison of $P_i(h)-hA_i(h)$ for $i=3,4$. Both follow from the same auxiliary function.

For $-1\leq z\leq1$, set
\[
 S(z)=B(z)-zD(z),
\]
and, for $0\leq u\leq1$, define
\begin{equation}\label{eq:J}
 J(u)=2S(u)-S(2u-1)-\frac\pi2.
\end{equation}

\begin{lemma}\label{lem:J}
For $0<u\leq1$, 
\begin{equation}\label{eq:J-bounds}
 J(u)\geq\gamma_\lambda>0,
 \qquad J(u)-uJ'(u)\geq\gamma_\lambda.
\end{equation}
\end{lemma}

\begin{proof}
For $-1<z<1$, differentiation gives
\[
 D'(z)=\frac{z}{\sqrt{1-z^2}}
       -\frac{z}{\lambda\sqrt{\lambda^2-z^2}},
 \qquad
 B'(z)+D(z)=zD'(z).
\]
It follows that
\begin{equation}\label{eq:Sprime}
 S'(z)=-2D(z).
\end{equation}
The function $D$ is even, and $D'(z)>0$ for $0<z<1$. Indeed,
\[
 \lambda\sqrt{\lambda^2-z^2}>\sqrt{1-z^2}
 \qquad (|z|<1).
\]
In particular, $D'(z)<0$ when $-1<z<0$.

The trigonometric identities in~\eqref{eq:BD} give
\[
 S(-z)=\lambda\pi-S(z),\qquad S(0)=\frac{\lambda\pi}{2}.
\]
Consequently,
\begin{equation}\label{eq:Jends}
 J(0)=J(1)=S(1)-\frac\pi2
 =\lambda\arccos(1/\lambda)-\sqrt{1-\lambda^{-2}}
 =\gamma_\lambda.
\end{equation}
To see that this number is positive, put $\theta=\arccos(1/\lambda)\in(0,\pi/2)$. Then
\[
 \gamma_\lambda=\lambda\theta-\sin\theta>0,
\]
since $\sin\theta<\theta<\lambda\theta$.

Using~\eqref{eq:Sprime}, we compute
\begin{equation}\label{eq:Jprime}
 J'(u)=4\bigl(D(2u-1)-D(u)\bigr),\qquad 0<u<1.
\end{equation}
Since $D$ is even and strictly increasing on $[0,1]$, this derivative is positive for $0<u<1/3$ and negative for $1/3<u<1$. Together with~\eqref{eq:Jends}, this proves the first inequality in~\eqref{eq:J-bounds}.

Let $H(u)=J(u)-uJ'(u)$. For $1/3\leq u\leq1$, the sign of $J'$ gives
\[
 H(u)\geq J(u)\geq\gamma_\lambda.
\]
For $0<u<1/3$, we have $2u-1<0$, and hence
\[
 J''(u)=8D'(2u-1)-4D'(u)<0.
\]
Therefore $H'(u)=-uJ''(u)>0$ on this interval. Formula~\eqref{eq:Jprime} extends continuously to $u=0$ with a finite value, so
\[
 \lim_{u\downarrow0}H(u)=J(0)=\gamma_\lambda.
\]
This proves the second inequality. All endpoint statements follow by continuity; no derivative of a square root is taken at $u=1$.
\end{proof}

\begin{proposition}\label{prop:comparison}
For every $\lambda>1$ and $0<h\leq1$,
\begin{align}
 A_4(h)-A_3(h)&\geq\frac{\gamma_\lambda}{h^2}>0,\label{eq:area-comparison}\\
 \bigl(P_4(h)-hA_4(h)\bigr)
 -\bigl(P_3(h)-hA_3(h)\bigr)
 &\geq\frac{\gamma_\lambda}{h}>0.\label{eq:support-comparison}
\end{align}
\end{proposition}

\begin{proof}
Let $q=2h-1$. Substituting~\eqref{eq:three}--\eqref{eq:four} gives
\begin{align*}
 h\bigl(P_4(h)-hA_4(h)\bigr)&=2S(h),\\
 h\bigl(P_3(h)-hA_3(h)\bigr)&=S(q)+\pi/2.
\end{align*}
Thus the difference of these two expressions is $J(h)$. A second direct calculation, using~\eqref{eq:Jprime}, yields
\begin{align*}
 h^2\bigl(A_4(h)-A_3(h)\bigr)
 &=2B(h)-B(q)-\pi/2+2hD(h)-(2h+1)D(q)\\
 &=J(h)-hJ'(h).
\end{align*}
The two claims now follow from Lemma~\ref{lem:J}.
\end{proof}

\section{The equal-area comparison}
We next record the variational identity needed to pass from the same-curvature inequalities to an equal-area competitor.

\begin{lemma}\label{lem:variation}
For $0<s<1,$ we have that
\begin{equation}
 P_3'(s)=sA_3'(s),
\end{equation}
and $A_3(s)\longrightarrow+\infty$ as $s\downarrow0$.
\end{lemma}

\begin{proof}
This follows from a direct computation and using the fact that $$B'(q)+D(q)=qD'(q).$$
\end{proof}

\begin{proof}[Proof of Theorem~\ref{thm:main}]
Fix $0<h\leq1$ and let $V=A_4(h)$. By~\eqref{eq:area-comparison},
\[
 A_3(h)<V.
\]
Choose $0<\varepsilon<h$ sufficiently small that $A_3(\varepsilon)>V$. The intermediate value theorem shows that
\[
 Z=\{s\in[\varepsilon,h]:A_3(s)=V\}
\]
is nonempty. It is compact, so it has a largest element $r$. The strict inequalities at the endpoints imply
\[
 0<\varepsilon<r<h\leq1.
\]
Moreover,
\begin{equation}\label{eq:last-root}
 A_3(s)<V\qquad\text{for }r<s\leq h.
\end{equation}
Indeed, equality would contradict the choice of $r$. If $A_3(s)>V$ at some such point, continuity and $A_3(h)<V$ would produce an equal-area root to its right, again contradicting maximality.

Define
\[
 Q(s)=P_3(s)+s\bigl(V-A_3(s)\bigr),\qquad r\leq s\leq h.
\]
By Lemma~\ref{lem:variation},
\[
 Q'(s)=V-A_3(s)>0\qquad (r<s<h).
\]
Since $Q$ is continuous on $[r,h]$, it follows that
\[
 P_3(r)=Q(r)<Q(h)
 =P_3(h)+h\bigl(A_4(h)-A_3(h)\bigr).
\]
Combining this inequality with~\eqref{eq:support-comparison} gives
\[
 P_4(h)-P_3(r)>\frac{\gamma_\lambda}{h}>0.
\]
The geometric identities~\eqref{eq:four-realized} and~\eqref{eq:three-realized} now give the desired competitor $E_r$. When $h=1$, the same proof applies because differentiation is used only on $(r,1)$ and $Q$ is continuous at $1$.
\end{proof}


\begin{thebibliography}{9}
\bibitem{CMV}
A. Ca\~nete, M. Miranda Jr.\ and D. Vittone,
\emph{Some isoperimetric problems in planes with density},
J. Geom. Anal. \textbf{20} (2010), 243--290.

\bibitem{DG1} 
E. De Giorgi, \emph{Su una teoria generale della misura (r-1)-dimensionale in uno spazio ad r dimensioni}, Ann. Mat. Pura Appl. (4) \textbf{36} (1954), 191--213. 

\bibitem{DG2}
\emph{E. De Giorgi, Sulla proprietà isoperimetrica dell'ipersfera, nella classe degli insiemi aventi frontiera orientata di misura finita}, Atti Accad. Naz. Lincei Mem. Cl. Sci. Fis. Mat. Nat. Sez. I (8) \textbf{5} (1958), 33–-44.

\bibitem{Borell}
C. Borell, \emph{The Brunn-Minkowski inequality in Gauss space}, Invent. Math. \textbf{30} (1975), no. 2, 207--216. 

\bibitem{RCBM}
C. Rosales, A. Cañete, V. Bayle, F. Morgan, \emph{On the isoperimetric problem in Euclidean space with density}, Calc. Var. Partial Differential Equations \textbf{31} (2008), no. 1, 27--46.

\bibitem{Chambers}
G. R. Chambers, \emph{Proof of the Log-Convex Density Conjecture}, J. Eur. Math. Soc. \textbf{21}, issue 8, 2301--2332

\bibitem{Repository}
\emph{CMV Strip-Density Project}, in the
\href{https://github.com/RosaRojacr/agentic-lean-math-assistant/tree/5e5dcd3d15df41a7c8051dd58473928fbe3ccc0c/projects/cmv-strip-density}{Agentic Lean Math Assistant repository},
snapshot \texttt{5e5dcd3d15df}.
See the modules \href{https://github.com/RosaRojacr/agentic-lean-math-assistant/blob/5e5dcd3d15df41a7c8051dd58473928fbe3ccc0c/projects/cmv-strip-density/proof/SameCurvatureArea.lean}{\texttt{SameCurvatureArea.lean}} and
\href{https://github.com/RosaRojacr/agentic-lean-math-assistant/blob/5e5dcd3d15df41a7c8051dd58473928fbe3ccc0c/projects/cmv-strip-density/proof/UniversalStationaryPair.lean}{\texttt{UniversalStationaryPair.lean}}.
\end{thebibliography}
\end{document}